\pdfoutput=1

\documentclass{amsart}

\usepackage{amsmath}
\usepackage{amssymb}
\usepackage{amsthm}

\newtheorem{theorem}{Theorem}[section]
\newtheorem{lemma}[theorem]{Lemma}
\newtheorem{proposition}[theorem]{Proposition}
\newtheorem{corollary}[theorem]{Corollary}

\newtheorem{conjecture}{Conjecture}
\newtheorem{thm}[conjecture]{Theorem}

\theoremstyle{definition}
\newtheorem{definition}[theorem]{Definition}

\theoremstyle{remark}
\newtheorem{remark}[theorem]{Remark}

\usepackage{hyperref}

\newcommand{\bk}{k}
\newcommand{\ideal}{\triangleleft}
\newcommand{\Spec}{\text{\normalfont Spec}\,}
\newcommand{\Sym}{\text{\normalfont Sym}}
\newcommand{\Set}{\text{\normalfont Set}}
\newcommand{\Amod}{A\text{\normalfont -mod}}
\newcommand{\Alg}{\bk \text{\normalfont -Alg}}

\newcommand{\Defemb}{\text{\normalfont Def}_{\text{\normalfont emb}}}
\newcommand{\hilb}{\mathcal H}

\newcommand{\Mor}{\text{\normalfont Mor}}
\newcommand{\Hom}{\text{\normalfont Hom}}
\newcommand{\Sch}{\text{\normalfont Sch}}

\newcommand{\m}{\mathfrak m}
\newcommand{\ma}{\normalfont\textbf{a}}
\newcommand{\mb}{\normalfont\textbf{b}}
\newcommand{\mx}{\normalfont\textbf{x}}

\newcommand{\Ann}{\text{\normalfont Ann}}
\newcommand{\Apolar}{\text{\normalfont Apolar}}

\newcommand{\op}{\text{\normalfont op}}

\newcommand{\Sa}{S}
\newcommand{\Sb}{S'}
\newcommand{\Sc}{S''}

\newcommand{\Ia}{I}
\newcommand{\Ib}{I'}
\newcommand{\Ic}{I''}

\newcommand{\Ja}{J}
\newcommand{\Jb}{J'}
\newcommand{\Jc}{J''}

\newcommand{\Fa}{F}
\newcommand{\Fb}{F'}
\newcommand{\Fc}{F''}

\title[New elementary components of the Gorenstein locus]{New elementary components of the Gorenstein locus of the Hilbert scheme of points}
\author{Robert Szafarczyk}
\date{4/06/2022}
\thanks{The author would like to thank Joachim Jelisiejew for introducing him to the subject as well as for many helpful suggestions. Without his support this article would not exist. The author was financed by the Ministry of Science and Higher Education of the Republic of Poland project ``Szkoła Orłów", project number 500-D110-06-0465160}

\usepackage{microtype}
\begin{document}

\begin{abstract}
We construct new explicit examples of nonsmoothable Gorenstein algebras with Hilbert function \((1,n,n,1)\). This gives a new infinite family of elementary components in the Gorenstein locus of the Hilbert scheme of points and solves the cubic case of Iarrobino's conjecture.
\end{abstract}

\maketitle

\tableofcontents

\section{Introduction}

Hilbert schemes of points were first constructed by Grothendieck in 1960-61 \cite{grothendieck}. Since then they have found many applications, notably in combinatorics \cite{comb} and in constructing hyperk\"ahler manifolds \cite{beauville}. Hilbert schemes of points also appear in complexity theory while studying tensor and border ranks \cite{landsberg}. One of the more important results of the theory is that by Fogarty stating that the Hilbert scheme of points of a smooth, irreducible surface is itself smooth and irreducible \cite{fogarty}. The Hilbert scheme of points for three- and higher-dimensional varieties is singular and not well understood.

The topology of the Hilbert scheme of points is still poorly understood and finding its irreducible components remains a challenge. The building blocks for them are the elementary components, those parametrizing subschemes with one-point support. Points of the Hilbert scheme of points corresponding to Gorenstein zero-dimensional subschemes form an open set, called the \emph{Gorenstein locus}. Few components of this locus are known. Additionally, the smooth points of these components are often not explicitly given. Explicit points outside of the smoothable component are of interest in applications to tensors \cite{tensors} and in computations.

Let \(S=\bk[\alpha_1,...,\alpha_n]\) and \(P=\bk[x_1,...,x_n]\) be polynomial rings of \(n\) variables over a field \(\bk\) of characteristic \(0\). There is an action of \(S\) on \(P\) defined as follows
\[\alpha_1^{b_1}\alpha_2^{b_2}...\,\alpha_n^{b_n} \cdot x_1^{a_1}x_2^{a_2}...\,x_n^{a_n} = \begin{cases} x_1^{a_1-b_1}x_2^{a_2-b_2}...\,x_n^{a_n-b_n} & \text{if}\ \forall_i\, a_i\geq b_i \\ 0 & \text{otherwise.} \end{cases}\]
This action is called \emph{contraction} and is somewhat similar to differentiation. For any polynomial \(f\in P\) the set \(\Ann(f)=\{s\in S : s\cdot f = 0\}\) forms a homogeneous ideal of \(S\), called the \emph{apolar ideal}. In this paper, we present the following result.

\begin{thm}
If \(n\geq 6\) (except for \(n=7\)), then for a general \(f\in P\) homogeneous of degree \(3\) the ideal \(\Ann(f)\) is a smooth point of an elementary component of the Hilbert scheme.
\end{thm}

This is a corollary of the following theorem. The \emph{apolar algebra} \(S/\Ann(f)\), denoted \(\Apolar(f)\), with Hilbert function \((1,n,n,1)\) is said to satisfy the small tangent space condition if the \(\bk\)-algebra \(S/\Ann(f)^2\) has the smallest Hilbert function possible.

\begin{thm} \label{thm}
If \(n\geq 6\) (except for \(n=7\)), then for a general \(f\in P\) homogeneous of degree \(3\) the apolar algebra \(\Apolar(f)\) satisfies the small tangent space condition.
\end{thm}

Loosely speaking, Theorem \ref{thm} asserts that \(\Apolar(f)\) has only trivial deformations of second order. For why it is false when \(n\leq 5\) or \(n=7\) see \cite{less_five} and \cite{seven}.

The characteristic \(0\) assumption can be removed for \(n\geq 18\), see Theorem \ref{thm_main}. We believe this to be true for all \(n\), but small \(n\) would probably require a direct verification. We make this verification on computer for characteristics 0,2, and 3. For \(n\) less than 13 this was also done by Iarrobino and Kanev \cite[Lemma 6.21]{iarrobino}.

Summing up, we show Theorem \ref{thm} to hold for all characteristics when \(n\geq 18\) and for characteristics 0,2, and 3 in general. This resolves the following conjecture, posed by Iarrobino and Kanev, in the case \(d=3\).

\begin{conjecture}[\cite{iarrobino}, Conjecture 6.30] \label{conj}
Let \(d\) be an odd integer. If one of the following conditions holds
\begin{enumerate}
\item \(n=4\) and \(d\geq 15\),
\item \(n= 5\) and \(d\geq 5\),
\item \(n\geq 6\) and \(d\geq 3\) (except for \((n,d)=(7,3)\)),
\end{enumerate}
then for a general \(f\in P\) homogeneous of degree \(d\) the apolar algebra \(\Apolar(f)\) satisfies the small tangent space condition.
\end{conjecture}

For \(d>3\) essentially nothing is known.

In order to prove Theorem \ref{thm} it suffices to give, for every \(n\), a single example of a polynomial whose apolar algebra satisfies the small tangent space condition. In our proof, we give three rather simple ones covering all \(n\) greater than 8. For \(n\) divisible by 3 we take the following polynomial
\[F=\sum_{i=1}^m a_ib_ic_i + a_ia_{i+1}^2 + b_ib_{i+1}^2 + c_ic_{i+1}^2.\]
As a consequence, we provide an explicit description of a smooth point of an elementary component of the Hilbert scheme. Since the apolar ideals associated to our examples admit a set of generators consisting only of monomials and binomials they are also convenient from a computational point of view. Moreover, for \(n\geq 18\), our proof does not use any computer computations. This is important in complexity theory, where structure tensors of such algebras correspond to 1-generic tensors \cite[Section 5.6.1]{landsberg}.

We begin, in chapter \ref{chap_prelim}, by giving all necessary background such as contraction, apolar algebras, and Gorenstein rings. It is also there where we compute the tangent space to the Hilbert scheme and present equivalent descriptions of the small tangent space condition. Then, in chapter \ref{chp_theorem}, we prove Theorem \ref{thm} for sufficiently large \(n\). Small \(n\) are taken care of in chapter \ref{chap_computer} where we verify them on a computer.

\section{Preliminaries} \label{chap_prelim}

This chapter introduces all notions related to our study. In section \ref{sec_hilbert}, we define the Hilbert scheme and describe its tangent space. In section \ref{sec_contraction}, we introduce apolar algebras and divided power rings associated to polynomial rings. In section \ref{sec_gorenstein}, we introduce the dualizing functor \((-)^{\vee}\) and define zero-dimensional Gorenstein local rings. In section \ref{sec_small_tan}, we define the small tangent space condition and relate it to the tangent space of the Hilbert scheme. Finally, in section \ref{sec_elementary}, we give a link between the small tangent space condition and smooth points on elementary components of the Hilbert scheme.

\subsection{Hilbert scheme} \label{sec_hilbert}
\hfill

In this section, we introduce the notion of deformation. The deformation functor turns out to be representable by a scheme, called the Hilbert scheme of points.

Let \(\bk\) be a field. Given two \(\bk\)-algebras \(S\) and \(A\), we write \(S_A\) for the ring \(S\otimes_{\bk} A\) treated as an \(A\)-algebra.

\begin{definition}
Let \(S\) be a fixed, finitely generated \(\bk\)-algebra. The \emph{embedded deformation functor} \(\Defemb\colon\Alg\to \Set\) assigns to a \(\bk\)-algebra \(A\) the set of isomorphism classes of ideals \(I\ideal S_A\) such that \(S_A/ I\) is a locally free \(A\)-module of finite rank. To a morphism \(A\to B\) of \(\bk\)-algebras the functor \(\Defemb\) assigns the function taking \(I\in \Defemb(A)\) to \(I S_B\in \Defemb(B)\).
\end{definition}

We consider the following theorem as the definition of the Hilbert scheme.

\begin{theorem}[\cite{haiman}, Theorem 1.1]
Let \(S\) be a fixed, finitely generated \(\bk\)-algebra. Then, there exists a finite type \(\bk\)-scheme \(\hilb\), called the Hilbert scheme of points, representing the deformation functor \(\Defemb\) in the sense that there is an isomorphism of sets \(\Defemb(A)\cong\Mor_{\Sch}(\Spec A,\hilb)\) natural in \(A\).
\end{theorem}

Note that \(\Defemb(\bk)\), and hence \(\hilb(\bk)\), is the set of ideals \(I\ideal S\) such that \(\dim_{\bk}S/I\) is finite. Since \(S/I\) is Noetherian \(\dim_{\bk}S/I\) being finite is equivalent to \(S/I\) being zero-dimensional.

\begin{theorem}[\cite{stromme}, Theorem 10.1] \label{thm_tangent}
Let \(S\) be a finitely generated \(\bk\)-algebra, and let \(\hilb\) be its associated Hilbert scheme. For an ideal \(I\triangleleft S\) such that \(\dim_{\bk}S/I\) is finite, hence for a rational point of \(\hilb\), the tangent space of \(\hilb\) at \(I\) is isomorphic to \(\Hom_S(I, S/I)\).
\end{theorem}

\subsection{Apolar algebras} \label{sec_contraction}
\hfill

In this section, following \cite{jelisiejew}, we introduce the notion of apolar algebra. This is the easiest to construct and, in the case of zero-dimensional, graded local rings, the only example of a Gorenstein ring (see Theorem \ref{thm_gor_class}).

Consider a polynomial ring \(S=\bk[\alpha_1,..,\alpha_n]\) over a field \(\bk\). Recall that \(S\) is a graded \(\bk\)-algebra with the ideal \(S_+\) equal to \((\alpha_1,...,\alpha_n)\). We denote by \(S^{\vee}\) the \(S\)-module \(\Hom_{\bk}(S,\bk)\) of \(\bk\)-linear functionals on \(S\). Let \(\langle-,-\rangle\colon S\times S^{\vee}\to \bk\) be the natural map given by evaluation.

\begin{definition}
Let \(P\) be the submodule \(\{f\in S^{\vee} \colon \forall_{N\gg 0}\,\langle (S_+)^N,f\rangle=0\}\) of \(S^{\vee}\). The induced action of \(S\) on \(P\) is called \emph{contraction}.
\end{definition}

We now give a more concrete description of contraction. If \(\ma=(a_1,a_2,...,a_n)\) is a multi-index, we write \(\alpha^{\ma}\) for the monomial \(\alpha_1^{a_1}\alpha_2^{a_2}...\,\alpha_n^{a_n}\in S\). For every multi-index \(\ma\) there is a unique functional \(\mx^{[\ma]}\in P\) dual to \(\alpha^{\ma}\) in the sense that for all multi-indices \(\mb\) we have
\[\langle \alpha^{\mb}, \mx^{[\ma]}\rangle=\begin{cases} 1 & \text{if}\ \ma = \mb \\ 0 & \text{otherwise}. \end{cases}\]

Note that \(\mx^{[\cdot]}\) form a \(\bk\)-basis of \(P\). The quantity \(\sum \ma:\equiv\sum a_i\) is called the \emph{degree} of \(\mx^{[\ma]}\). An element \(f\in P\) is called \emph{homogeneous} of degree \(d\) if \(f\) is contained in \(\text{span}_{\bk}(\mx^{[\ma]}\colon\sum\ma=d)\). Contraction behaves on the basis as follows
\[\alpha^{\mb}\cdot\mx^{[\ma]} = \begin{cases} \mx^{[\ma-\mb]} & \text{if}\ \ma \geq \mb \ (\forall_i\, a_i\geq b_i) \\ 0 & \text{otherwise}. \end{cases}\]

Though we do not need this, we can equip \(P\) with a ring structure. Multiplication on \(P\) is given by the formula
\[\mx^{[\ma]} \mx^{[\mb]} = \binom{\ma+\mb}{\ma} \mx^{[\ma+\mb]}\]
where \(\binom{\ma+\mb}{\ma}=\prod \binom{a_i+b_i}{a_i}\). In this way, \(P\) is a \emph{divided power ring}.

\begin{definition}
Let \(f\in P\), and let \(\Ann(f)\) denote the ideal \(\{s\in S \colon s\cdot f = 0\}\) of \(S\). The \(\bk\)-algebra \(S/\Ann(f)\) is called the \emph{apolar algebra} of \(f\), and is denoted \(\Apolar(f)\).
\end{definition}

\subsection{Zero-dimensional Gorenstein local rings} \label{sec_gorenstein}
\hfill

Throughout this section let \((A,\m,\bk)\) be a zero-dimensional, finitely generated local \(k\)-algebra. We denote by \(\Amod\) the category of finitely generated modules over \(A\).

We recall basic definitions and properties concerning zero-dimensional Gorenstein rings following \cite[Chapter 21]{eisenbud}.

\begin{definition}
A functor \(E:\Amod^{\op}\to\Amod\) is called \emph{dualizing} if \(E^2\cong 1\).
\end{definition}

\begin{proposition}[\cite{eisenbud}, Proposition 21.1] \label{prop_dualizing}
If \(E\colon \Amod^{\op}\to \Amod\) is dualizing, then there is an isomorphism of functors \(E\cong \Hom_A(-,E(A))\). Moreover, up to isomorphism there exists at most one dualizing functor on \(\Amod\).
\end{proposition}

Consider the functor \((-)^{\vee}:\equiv \Hom_{\bk}(-,\bk)\). For an \(A\)-module \(M\), the vector space \(\Hom_{\bk}(M,\bk)\) naturally forms an \(A\)-module with the \(A\)-action given by
\[(a\cdot \varphi)(m) = \varphi(am)\]
where \(\varphi\in M^{\vee}\), \(m\in M\), and \(a\in A\). Therefore, we can view \((-)^{\vee}\) as a functor \(\Amod^{\op}\to\Amod\).

\begin{proposition}[\cite{eisenbud}, Section 21.1] \label{prop_vee}
The functor \((-)^{\vee}\) is dualizing.
\end{proposition}

Combining Propositions \ref{prop_dualizing} and \ref{prop_vee} shows that up to isomorphism there exists a unique dualizing functor on \(\Amod\).

\begin{definition}
We say that \(A\) is \emph{Gorenstein} if \(A^{\vee}\cong A\).
\end{definition}

If \(A\) is Gorenstein, then in view of Proposition \ref{prop_dualizing} we have an isomorphism of functors \(\Hom_A(-,A)\cong\Hom_A(-,A^{\vee})\cong (-)^{\vee}\). In particular \(\Hom_A(-,A)\) is dualizing. Conversely, if \(\Hom_A(-,A)\) is dualizing, then, by Proposition \ref{prop_dualizing} and Yoneda lemma, \(A\) is isomorphic to \(A^{\vee}\), so \(A\) is Gorenstein.

We have the following characterization of zero-dimensional Gorenstein rings.

\begin{proposition}[\cite{eisenbud}, Proposition 21.5] \label{prop_gor_class}
Let \((A,\m,\bk)\) be a zero-dimensional, finitely generated local \(k\)-algebra. Then, the following conditions are equivalent.
\begin{enumerate}
\item \(A\) is Gorenstein.
\item \(A\) is injective as an \(A\)-module.
\item The annihilator of the maximal ideal \(\Ann(\m)\subset A\) is one dimensional.
\item \(\Hom_A(-,A)\) is dualizing.
\end{enumerate}
\end{proposition}

\subsection{Small tangent space condition} \label{sec_small_tan}
\hfill

In this section, we introduce the small tangent space condition and better describe the tangent space of the Hilbert scheme associated to a polynomial ring. We also prove Proposition \ref{prop_general}, which is needed in the proof of Theorem \ref{thm_main}.

As in section \ref{sec_contraction}, let \(S\) be a polynomial ring of \(n\) variables over a field \(\bk\), and let \(P\) be its associated divided power ring. We denote by \(\hilb\) the Hilbert scheme associated to \(S\). Recall that \(\Apolar(f)= S/\Ann(f)\). If \(I=\Ann(f)\) we write \(S/I\) and \(\Apolar(f)\) interchangeably.

When saying that a graded module \(M\) has Hilbert function \((h_0,h_1,...,h_j)\), \(h_i\in \mathbb N\) we mean that \(H(M)_i\) is equal to \(h_i\) for \(i\in\{0,1,...,j\}\) and that \(H(M)_i\) is equal to \(0\) for \(i\notin\{0,1,...,j\}\). For example, \(S\) has Hilbert function \((1,n,\binom{n+1}{2},\binom{n+2}{3},...)\). We denote a shift of gradation in square brackets, so \(M[d]_i=M_{d+i}\).

\begin{theorem}[\cite{iarrob}, Lemma 1.2 and Theorem 1.5] \label{thm_gor_class}
For every nonzero \(f\in P\) homogeneous of degree \(d\) the apolar algebra \(\Apolar(f)\) is a graded zero-dimensional Gorenstein local ring. Moreover, there is a graded isomorphism \(\Apolar(f)\cong \Apolar(f)^{\vee}[-d]\).

For every homogeneous ideal \(I\triangleleft S\), if \(S/I\) is a zero-dimensional Gorenstein local ring, then there exists homogeneous \(f\in P\) such that \(I=\Ann(f)\).
\end{theorem}

We can now describe the tangent space of the Hilbert scheme more concretely.

\begin{proposition} \label{prop_hilb_I}
Let \(f\in P\) be homogeneous of degree \(d\), and let \(I=\Ann(f)\). Then, the tangent space \(T_I\hilb\) is isomorphic as a graded module to \((I/I^2)^{\vee}[-d]\).
\end{proposition}

\begin{proof}
By Theorem \ref{thm_tangent} the tangent space to \(\hilb\) at \(I\) is isomorphic to \(\Hom_{S}(I,S/I)\). By the tensor-hom adjunction and the isomorphism \(I\otimes_S S/I\cong I/I^2\) we get
\[\Hom_{S}(I,S/I)\cong \Hom_{S/I}(I/I^2,S/I).\]
Then, Theorem \ref{thm_gor_class} yields
\[\Hom_{S/I}(I/I^2,S/I)\cong \Hom_{S/I}(I/I^2,(S/I)^{\vee})[-d].\]
Now, again by the tensor-hom adjunction, we obtain
\[\Hom_{S/I}(I/I^2,(S/I)^{\vee})[-d]\cong (I/I^2)^{\vee}[-d].\]
Hence \(T_I\hilb\cong (I/I^2)^{\vee}[-d]\)
as required.
\end{proof}

From now on, since we are mainly interested in degree 3 homogeneous elements of \(P\), we reduce ourselves to this special case.

Let \(f\in P\) be homogeneous of degree 3, and let \(I=\Ann(f)\). Since \(S_{\geq 4}\) is contained in \(I\) the Hilbert function of \(S/I\) can be nonzero only in degrees \(0, 1, 2\), and 3. Moreover, since \((S/I)^{\vee}\cong (S/I)[3]\) the Hilbert function is symmetric in the sense that \(H(S/I)_0=H(S/I)_3\) and \(H(S/I)_1=H(S/I)_2\). Clearly, \(H(S/I)_0=1\) and \(H(S/I)_1\leq n\). In view of the following proposition, the case \(H(S/I)_1<n\) might be considered degenerate.

\begin{proposition}[\cite{iarrobino}, Proposition 3.12] \label{prop_gen_hilbert_fun}
There is an open, dense subset \(U\) of the space of cubics \(\Spec \Sym (P_3)^{\vee}\) such that for all rational points \(f\in U(\bk)\) the Hilbert function of \(\Apolar(f)\) is \((1,n,n,1)\).
\end{proposition}

\begin{proposition} \label{prop_low_bound}
Let \(f\in P\) be homogeneous of degree \(3\), and let \(I=\Ann(f)\). If \(\Apolar(f)\) has Hilbert function \((1,n,n,1)\), then \(H(S/I^2)_4\geq n\).
\end{proposition}

\begin{proof}
This follows from \cite[Lemma 3.4]{vsp}.
\end{proof}

\begin{definition}
Let \(f\in P\) be homogeneous of degree \(3\), and let \(I=\Ann(f)\). Then, we say that \(\Apolar(f)\) satisfies the \emph{small tangent space condition} if \(\Apolar(f)\) has Hilbert function \((1,n,n,1)\) and \(H(S/I^2)_4=n\), \(H(S/I^2)_5=0\).
\end{definition}

\begin{proposition} \label{prop_tan_I}
Let \(f\in P\) be homogeneous of degree \(3\). Then, \(\Apolar(f)\) satisfies the small tangent space condition if and only if the tangent space of \(\hilb\) at \(I=\Ann(f)\) has Hilbert function \(n,\binom{n+2}{3}-1,\binom{n+1}{2}-n\) in degrees \(-1,0,1\) respectively, and \(0\) elsewhere.
\end{proposition}

\begin{proof}
First suppose that \(\Apolar(f)\) satisfies the small tangent space condition. In view of Proposition \ref{prop_hilb_I} we need to compute the Hilbert function of \(I/I^2\). The ring \(S/I\) has Hilbert function \((1,n,n,1)\), so \(I_{\leq 1}=0\). It follows that \(S/I^2\) is all of \(S\) in degrees \(0,1,2\), and 3. Furthermore, since \(S/I\) satisfies the small tangent space condition \(S_5\) is contained in \(I^2\), so \(S_{\geq 6}\) is contained in \(I^2\) as well, which means that \(H(S/I^2)_{\geq 6}\) is \(0\). Thus, the Hilbert function of \(S/I^2\) is \((1,n,\binom{n+1}{2},\binom{n+2}{3},n)\). Furthermore, since \(S/I\) has Hilbert function \((1,n,n,1)\), we get that the Hilbert function of \(I/I^2\) is equal to \((0,0,\binom{n+1}{2}-n,\binom{n+2}{3}-1,n)\), hence the tangent space \(T_I\hilb\cong (I/I^2)^{\vee}[-3]\) has the desired Hilbert function.

Now suppose that \(T_I\hilb\) has the given Hilbert function. Then, \(I/I^2\) has Hilbert function \((0,0,\binom{n+1}{2}-n,\binom{n+2}{3}-1,n)\). Since \(S_{\geq4}\subset I\) this means that \(H(S/I^2)_4=n\) and \(H(S/I^2)_5=0\). Moreover, since \(I^2\) is not all of \(S\) we have \(H(I)_0=0\), so \(H(I^2)_1=0\). Thus, \(H(S/I)_0=1\) and \(H(S/I)_1=n\), which since \(S/I\) is Gorenstein implies that \(S/I\) has Hilbert function \((1,n,n,1)\) as required.
\end{proof}

\begin{corollary} \label{cor_small_tan_hilb}
Let \(f\in P\) be homogeneous of degree 3 such that \(\Apolar(f)\) has Hilbert function \((1,n,n,1)\). Then, \(\Apolar(f)\) satisfies the small tangent space condition if and only if the tangent space of \(\hilb\) at \(\Ann(f)\) has the smallest Hilbert function possible.
\end{corollary}

\begin{proposition} \label{prop_general}
Let \(f\in P\) be homogeneous of degree \(3\) such that \(\Apolar(f)\) satisfies the small tangent space condition. Then, for a general \(g\in P\) homogeneous of degree \(3\) the apolar algebra \(\Apolar(g)\) satisfies the small tangent space condition.
\end{proposition}

\begin{proof}
Let \(U\) be the open subscheme of \(\Spec\Sym(P_3)^{\vee}\) from Proposition \ref{prop_gen_hilbert_fun}. By \cite[Section 2.2]{vsp} there is a family \(Z\subset U\times \mathbb A^n_{\bk} \to U\) such that the fibre over \(f\in U(\bk)\) is \(\Spec\Apolar(f)\). Hence, the claim follows from Corollary \ref{cor_small_tan_hilb} and upper semi-continuity of rank for the quasicoherent sheaf \(I(Z)/I(Z)^2\).
\end{proof}

\subsection{Elementary components} \label{sec_elementary}
\hfill

We now describe the connection between the small tangent space condition and smooth points on elementary components of the Hilbert scheme.

As in section \ref{sec_small_tan}, we only consider degree 3 homogeneous elements of \(P\).

\begin{definition}
An irreducible component \(Z\) of the Hilbert scheme is called \emph{elementary} if for all rational points \(I\in Z(\bk)\) the \(S\)-module \(S/I\) is supported at a single point.
\end{definition}

\begin{proposition} \label{prop_elem_small}
Let \(f\in P\) be homogeneous of degree 3. If \(\Apolar(f)\) satisfies the small tangent space condition, then \(I=\Ann(f)\) is a smooth point of an elementary component of \(\hilb\).
\end{proposition}

\begin{proof}
Since \(\Apolar(f)\) satisfies the small tangent space condition, by Proposition \ref{prop_tan_I}, we have \(\dim_{\bk}\Hom_S(I,S/I)_{<0}=n\). Hence, by \cite[Theorem 4.5 and Corollary 4.7]{jelelem}, all irreducible components containing \(I\) are elementary. Smoothness follows from the discussion in \cite[Proof of Lemma 6.21]{iarrobino}.
\end{proof}

\section{Small tangent space condition in degree 3} \label{chp_theorem}

In this chapter, we prove Theorem \ref{thm_main}, which confirms Conjecture \ref{conj} in the case where \(d=3\) and \(n\geq 18\).

\begin{theorem} \label{thm_main}
Let \(S\) be a polynomial ring of \(n\) variables over a field \(\bk\), and let \(P\) be its associated divided power ring. If \(n\geq 18\), then for a general \(f\in P\) homogeneous of degree \(3\) the apolar algebra \(\Apolar(f)\) satisfies the small tangent space condition.
\end{theorem}

\begin{proof}
In view of Proposition \ref{prop_general} it suffices to find, for each \(n\geq 18\), a single \(f\in P\) such that \(\Apolar(f)\) satisfies the small tangent space condition. We divide the proof into three cases: \(n=3m\), \(n=3m+1\), and \(n=3m+2\), where \(m\geq 6\). They are resolved by Propositions \ref{prop_done_0}, \ref{prop_done_1}, and \ref{prop_done_2} respectively.
\end{proof}

\begin{corollary}
Let \(S\) be a polynomial ring of \(n\) variables over a field \(\bk\), and let \(P\) be its associated divided power ring. If \(n\geq 18\), then for a general \(f\in P\) homogeneous of degree \(3\) the ideal \(\Ann(f)\) is a smooth point of an elementary component of the Hilbert scheme.
\end{corollary}

\begin{proof}
This follows by combining Proposition \ref{prop_elem_small} and Theorem \ref{thm_main}.
\end{proof}

\subsection{Proof of Theorem \ref*{thm_main}; case \(n=3m\)}
\hfill

Let \(\Sa=\bk[a_i,b_i,c_i]_{i=1}^m\) be a polynomial ring of \(n=3m\) variables. Recall that we assume \(m\geq 6\). When writing indices we treat them modulo \(m\).

Consider the following polynomial
\[\Fa=\sum_{i=1}^m a_ib_ic_i + a_ia_{i+1}^2 + b_ib_{i+1}^2 + c_ic_{i+1}^2.\]
Also by \(\Fa\) we denote its dual element in the divided power ring associated to \(\Sa\).

Let \(\Ia\) be the smallest ideal such that the following remark holds.

\begin{remark}
For all \(x,y\in\{a,b,c\}\), and for all indices \(i,j\), if \(j\notin\{i-1,i,i+1\}\), then \(x_iy_j\in \Ia\).

For all \(x,y\in\{a,b,c\}\), and for all indices \(i,j\), if \(x\ne y\), then \(x_iy_{i+1}-x_jy_{j+1}\in \Ia\).

For all \(x,y\in\{a,b,c\}\), and for all indices \(i,j\), one of the following holds:
\begin{enumerate}
\item \(x_iy_j\) is contained in \(\Ia\).
\item \(x\ne y\) and \(j\in\{i-1,i+1\}\).
\item There exists an index \(k\) and \(p,q\in\{a,b,c\}\), \(p\ne q\), such that \(x_iy_j-p_kq_k\in \Ia\).
\end{enumerate}
\end{remark}

The polynomial \(\Fa\) is chosen such that \(\Ia\subset\Ann\,\Fa\).

Let \(\Ja\) denote the ideal \(\Ia^2 + \langle a_ia_{i+1}a_{i+2}^2,b_ib_{i+1}b_{i+2}^2,c_ic_{i+1}c_{i+2}^2 \,|\, i=1,...,m\rangle\). Note that \(\Fa,\Ia\), and \(\Ja\) are invariant under index translation and permutations of the set \(\{a,b,c\}\).

We want to show that \(\Apolar(\Fa)\) satisfies the small tangent space condition. The main part of the proof is checking that all polynomials of degree \(4\) are contained in \(\Ja\), hence that \(H(\Sa/\Ia^2)_4=n\).

In this section, we use the following notation. For polynomials \(Q,R\in \Sa\) we write \(Q\equiv R\) if \(Q\) is equal to \(R\) in \(\Sa/\Ia^2\).

\begin{lemma} \label{lem_index_0}
For all \(x,y,z,w,p,q\in\{a,b,c\}\), and for all indices \(i,j,k\),
\begin{enumerate}
\item there exists an index \(t\) such that \(x_iy_jz_tw_{t+1}\) is in \(\Ia^2\).
\item there exist indices \(s,t\) such that each of \(x_iy_{i+1}z_tw_{t+1}\), \(x_sy_{s+1}z_jw_{j+1}\), and \(x_sy_{s+1}z_tw_{t+1}\) is in \(\Ia^2\).
\item if \(j,k\in\{i-1,i,i+1\}\), then there exists an index \(s\) such that each of \(x_iy_jz_sw_{s+1}\) and \(z_sw_{s+1}p_kq_k\) is in \(\Ia^2\).
\end{enumerate}
\end{lemma}

\begin{proof}
Throughout the proof we use the assumption \(m\geq 6\).

We first prove (1). If \(j\notin\{i+1,i+2,i+3\}\), then we take \(t=i+2\), so that \(x_iy_jz_{i+2}w_{i+3}=(x_iw_{i+3})(y_jz_{i+2}) \equiv 0\). If \(j\in\{i+1,i+2,i+3\}\), then we take \(t=i-2\), so that \(x_iy_jz_{i-2}w_{i-1}=(x_iz_{i-2})(y_jw_{i-1}) \equiv 0\).

Now we show (2). If \(j\notin \{i+1,i+2,i+3\}\), then we take \(s=j+2\), \(t=j-2\), so that 
\begin{alignat*}{2}
x_iy_{i+1}z_{j-2}w_{j-1} & =(x_iz_{j-2})(y_{i+1}w_{j-1})\equiv 0\\
x_{j+2}y_{j+3}z_jw_{j+1}&=(x_{j+2}z_j)(y_{j+3}w_{j+1})\equiv 0\\
x_{j+2}y_{j+3}z_{j-2}w_{j-1}&=(x_{j+2}z_{j-2})(y_{j+3}w_{j-1})\equiv 0.
\end{alignat*}
If \(j\in\{i+1,i+2,i+3\}\), then we take \(s=j-2\), \(t=j+1\), so that
\begin{alignat*}{2}
x_iy_{i+1}z_{j+1}w_{j+2}&=(x_iz_{j+1})(y_{i+1}w_{j+2})\equiv 0\\
x_{j-2}y_{j-1}z_jw_{j+1}&=(x_{j-2}z_j)(y_{j-1}w_{j+1}) \equiv 0\\
x_{j-2}y_{j-1}z_{j+1}w_{j+2} &= (x_{j-2}z_{j+1})(y_{j-1}w_{j+2})\equiv 0.
\end{alignat*}

Finally, we prove (3). If \(k=i+1\), then we take \(s=i-3\), so that
\begin{alignat*}{2}
x_iy_jz_{i-3}w_{i-2}&=(x_iw_{i-2})(y_jz_{i-3})\equiv 0\\
z_{i-3}w_{i-2}p_{i+1}q_{i+1}&=(z_{i-3}p_{i+1})(w_{i-2}q_{i+1})\equiv 0.
\end{alignat*}
If \(k\in\{i-1,i\}\), then we take \(s=i+2\), so that
\begin{alignat*}{2}
x_iy_jz_{i+2}w_{i+3}&=(x_iz_{i+2})(y_jw_{i+3})\equiv 0\\
z_{i+2}w_{i+3}p_{k}q_{k}&=(z_{i+2}p_{k})(w_{i+3}q_{k})\equiv 0.
\end{alignat*}
This finishes the proof.
\end{proof}

\begin{lemma} \label{lem_first}
For all \(x,y,z,w\in\{a,b,c\}\), \(z\ne w\), and for all indices \(i,j,k\), if either \(x_jy_k\in \Ia\) or \(x\ne y\) and \(k=j+1\), then the monomial \(z_iw_{i+1}x_jy_k\) is contained in \(\Ia^2\).
\end{lemma}

\begin{proof}
By symmetry we can assume \(z=a\), \(w=b\). Hence, we only need to examine monomials of the form \(a_ib_{i+1}x_jy_k\), where either \(x_jy_k\in \Ia\) or \(x\ne y\) and \(k=j+1\).

First consider the case \(x_jy_k\in \Ia\). By Lemma \ref{lem_index_0} there exists an index \(t\) such that \(a_tb_{t+1}x_jy_k\in \Ia^2\). Therefore, \(a_ib_{i+1}x_jy_k=(a_ib_{i+1}-a_tb_{t+1})(x_jy_k) + a_tb_{t+1}x_jy_k \equiv 0\).

Now consider the case where \(x\ne y\) and \(k=j+1\). Then, by Lemma \ref{lem_index_0}, there exist indices \(t\) and \(s\) such that \(a_tb_{t+1}x_jy_{j+1}\), \(a_ib_{i+1}x_sy_{s+1}\), and \(a_tb_{t+1}x_sy_{s+1}\) are in \(\Ia^2\). Therefore, \(a_ib_{i+1}x_jy_{j+1}= (a_ib_{i+1}-a_tb_{t+1})(x_jy_{j+1}-x_{s}y_{s+1}) + a_ib_{i+1}x_sy_{s+1} + a_tb_{t+1}x_jy_{j+1} - a_tb_{t+1}x_sy_{s+1} \equiv 0\).
\end{proof}

\begin{lemma} \label{lem_special_cases}
For all indices \(i\) and \(j\), the monomial \(a_ia_jb_ic_i\) is contained in \(\Ja\).
\end{lemma}

\begin{proof}
First assume that \(j\notin\{i-2,i-1,i,i+1\}\). Then, we can rewrite \(a_ia_jb_ic_i\) as follows.
\begin{alignat*}{2}
a_ia_jb_ic_i & = (a_ib_i-c_{i-1}c_i)(a_jc_i) + (a_jc_{i-1})(c_i^2-a_{i-1}b_{i-1}) + \\
 & + a_{i-1}a_jb_{i-1}c_{i-1} \equiv a_{i-1}a_jb_{i-1}c_{i-1}
\end{alignat*}
Hence, iterating this procedure, we get \(a_ia_jb_ic_i\equiv a_ja_{j+2}b_{j+2}c_{j+2}\). Therefore, we just need to examine monomials \(a_{i-2}a_ib_ic_i\), \(a_{i-1}a_ib_ic_i\), \(a_i^2b_ic_i\), and \(a_ia_{i+1}b_ic_i\).

Monomial \(a_{i-2}a_ib_ic_i\) can be rewritten in the following way.
\begin{alignat*}{2}
a_{i-2}a_ib_ic_i & = (a_ib_i-c_{i-1}c_i)(a_{i-2}c_i) + (a_{i-2}c_{i-1}-a_{i+2}c_{i+3})(c_i^2-a_{i-1}b_{i-1}) + \\
 & + a_{i-2}a_{i-1}b_{i-1}c_{i-1} + (a_{i+2}c_i)(c_ic_{i+3}) - (a_{i-1}a_{i+2})(b_{i-1}c_{i+3}) \equiv \\
 & \equiv a_{i-2}a_{i-1}b_{i-1}c_{i-1}
\end{alignat*}
Hence, we are reduced to considering \(a_{i-1}a_ib_ic_i\), \(a_i^2b_ic_i\), and \(a_ia_{i+1}b_ic_i\). Before we do so, we make some auxiliary computations.
\begin{alignat*}{2}
a_i^4 & = (a_i^2-b_{i-1}c_{i-1})^2-(b_{i-1}^2-b_{i-3}b_{i-2})(c_{i-1}^2-c_{i-3}c_{i-2})+ \\
 & - b_{i-1}^2c_{i-3}c_{i-2} - b_{i-3}b_{i-2}c_{i-1}^2 + b_{i-3}b_{i-2}c_{i-3}c_{i-2} + \\
 & + 2a_i^2b_{i-1}c_{i-1} \\
a_i^2a_{i+1}^2 & = (a_i^2-b_{i-1}c_{i-1})(a_{i+1}^2-b_ic_i) + a_i^2b_ic_i + (a_{i+1}b_{i-1})(a_{i+1}c_{i-1}) + \\
 & - b_{i-1}b_ic_{i-1}c_i
\end{alignat*}
Hence, Lemma \ref{lem_first} shows \(a_i^4\equiv 0\) and \(a_i^2a_{i+1}^2\equiv a_i^2b_ic_i\). We make further computations, where we assume \(j\notin\{i-1,i,i+1\}\).
\begin{alignat*}{2}
a_{i}^2b_{i}c_{i} & \equiv a_i^2a_{i+1}^2 = (a_{i+2}^2-a_{i}a_{i+1})^2 - a_{i+2}^4 + 2a_{i}a_{i+1}a_{i+2}^2 \equiv 2a_{i}a_{i+1}a_{i+2}^2 \in \Ja \\
a_ia_ja_{j+1}^2 & = (a_{i}a_{j})(a_{j+1}^2-a_{j-1}a_{j}) + a_ia_{j-1}a_j^2  \equiv a_ia_{j-1}a_{j}^2
\end{alignat*}
Iterating the last computation we get \(a_ia_ja_{j+1}^2\equiv a_ia_{i+1}a_{i+2}^2\).
We are now ready to rewrite \(a_{i-1}a_ib_ic_i\) and \(a_ia_{i+1}b_ic_i\).
\begin{alignat*}{2}
a_{i-1}a_ib_ic_i & = (a_{i-1}a_i-a_{i+1}^2)(b_ic_i-a_{i+1}^2)+ a_{i-1}a_ia_{i+1}^2 + \\
 & + (a_{i+1}b_i-a_{i+3}b_{i+2})(a_{i+1}c_i-a_{i-1}c_{i-2}) + (a_{i-1}a_{i+1})(b_ic_{i-2}) + \\
 & + (a_{i+1}a_{i+3})(b_{i+2}c_i) - (a_{i-1}a_{i+3})(b_{i+2}c_{i-2}) - a_{i+1}^4 \equiv \\
 & \equiv a_{i-1}a_{i}a_{i+1}^2 \in \Ja \\
a_ia_{i+1}b_ic_i & = (a_ia_{i+1}-a_{i+2}^2)(b_ic_i-a_{i+1}^2) + (a_ia_{i+1}-a_{i+2}^2)(a_{i+1}^2-a_{i-1}a_i) + \\
 & + a_{i-1}a_i^2a_{i+1} + a_{i+1}^2a_{i+2}^2 - (a_{i-1}a_{i+2})(a_ia_{i+2}) + (a_{i+2}b_i)(a_{i+2}c_i) + \\
 & -a_{i+1}^2a_{i+2}^2 \equiv a_{i+1}a_{i+2}a_{i+3}^2 \in \Ja
\end{alignat*}
This finishes the proof.
\end{proof}

\begin{lemma} \label{lem_i_plus_one}
For all \(x,y,z,w\in\{a,b,c\}\), \(z\ne w\), and for all indices \(i,j,k\), the monomial \(z_iw_{i+1}x_jy_k\) is contained in \(\Ja\).
\end{lemma}

\begin{proof}
By symmetry we can assume \(z=a\), \(w=b\). Hence, we only need to examine monomials of the form \(a_ib_{i+1}x_jy_k\).

Lemma \ref{lem_first} covers some of the cases. In any other there exist \(p,q\in\{a,b,c\}\), \(p\ne q\), and an index \(t\) such that \(x_jy_k-p_tq_t\in \Ia\). Furthermore, since \(x_jy_k\notin \Ia\) we have \(k,t\in\{j-1,j,j+1\}\), hence, by Lemma \ref{lem_index_0}, we can choose an index \(s\) such that both \(a_sb_{s+1}x_jy_k\) and \(a_sb_{s+1}p_tq_t\) are in \(\Ia^2\). Then, we have \(a_ib_{i+1}x_jy_k=(a_ib_{i+1}-a_sb_{s+1})(x_jy_k-p_tq_t) + a_ib_{i+1}p_tq_t + a_sb_{s+1}x_jy_k - a_sb_{s+1}p_tq_t \equiv a_{i}b_{i+1}p_tq_t\). Hence, to finish the proof it suffices to consider monomials of the form \(a_ib_{i+1}x_jy_j\) with \(x\ne y\).

If \(j\notin \{i-1,i,i+1,i+2\}\), then \(a_ib_{i+1}x_jy_j=(a_ix_j)(b_{i+1}y_j)\equiv 0\).

Since \(x\ne y\) one of them is not \(a\), say \(x\ne a\). Then, if \(j=i-1\), we can write \(a_ib_{i+1}x_{i-1}y_{i-1}\) as \(x_{i-1}a_iy_{i-1}b_{i+1}\), and since \(y_{i-1}b_{i+1}\in \Ia\), Lemma \ref{lem_first} applies showing \(x_{i-1}a_iy_{i-1}b_{i+1}\equiv 0\). Similarly, if \(j=i+2\), since one of \(x,y\) in not \(b\), say \(x\ne b\), we know, by Lemma \ref{lem_first}, that \(b_{i+1}x_{i+2}a_iy_{i+2}\equiv 0\).

In the case \(j=i+1\) we need to consider monomials \(a_ia_{i+1}b_{i+1}^2\), \(a_ia_{i+1}b_{i+1}c_{i+1}\), and \(a_ib_{i+1}^2c_{i+1}\). Monomial \(a_ia_{i+1}b_{i+1}c_{i+1}\) is a special case of Lemma \ref{lem_special_cases}. Others can be rewritten as follows.
\begin{alignat*}{2}
a_ia_{i+1}b_{i+1}^2 & = (a_ib_{i+1}-a_{i-2}b_{i-1})(a_{i+1}b_{i+1}-c_{i+2}^2) + a_ib_{i+1}c_{i+2}^2 + \\
 & + (a_{i-2}a_{i+1})(b_{i-1}b_{i+1}) - (a_{i-2}c_{i+2})(b_{i-1}c_{i+2}) \equiv 0 \\
a_ib_{i+1}^2c_{i+1} & = (a_ib_{i+1}-a_{i-2}b_{i-1})(b_{i+1}c_{i+1}-a_{i+2}^2) + a_ia_{i+2}^2b_{i+1} + \\
 & + (a_{i-2}b_{i+1})(b_{i-1}c_{i+1}) - (a_{i-2}a_{i+2})(b_{i-1}a_{i+2}) \equiv 0
\end{alignat*}
where \(a_ib_{i+1}c_{i+2}^2\) and \(a_ia_{i+2}^2b_{i+1}\) are in \(\Ia^2\) by Lemma \ref{lem_first}.

It remains to consider \(j=i\). We need to examine three monomials, \(a_i^2b_ib_{i+1}\), \(a_i^2b_{i+1}c_i\), and \(a_ib_ib_{i+1}c_i\). We can rewrite them as follows.
\begin{alignat*}{2}
a_i^2b_ib_{i+1} & = (a_i^2-a_{i-2}a_{i-1})(b_ib_{i+1}-b_{i+2}^2) + (a_ib_{i+2})^2 + \\ 
 & + (a_{i-2}b_i)(a_{i-1}b_{i+1}) - (a_{i-2}b_{i+2})(a_{i-1}b_{i+2}) \equiv 0 \\
a_i^2b_{i+1}c_i & = (a_i^2-a_{i-2}a_{i-1})(b_{i+1}c_i-b_{i+3}c_{i+2}) + (a_ib_{i+3})(a_ic_{i+2}) + \\
 & + (a_{i-2}c_i)(a_{i-1}b_{i+1}) - (a_{i-2}c_{i+2})(a_{i-1}b_{i+3}) \equiv 0
\end{alignat*}
Since, by symmetry, \(a_ib_ib_{i+1}c_i\) is a special case of Lemma \ref{lem_special_cases} the proof is finished.
\end{proof}

\begin{lemma} \label{lem_\Ia}
For all \(x,y,z,w\in\{a,b,c\}\), and for all indices \(i,j,k,t\) such that \(x_iy_j\in \Ia\), the monomial \(x_iy_jz_kw_t\) is in \(\Ja\).
\end{lemma}

\begin{proof}
If \(z_kw_t\) is contained in \(\Ia\) as well, then \(x_iy_jz_kw_t\in \Ia^2\). If \(z_kw_t\) is of the form \(p_sq_{s+1}\) for some index \(s\) and \(p,q\in\{a,b,c\}\), \(p\ne q\), then Lemma \ref{lem_i_plus_one} applies.

In any other case there exist \(p,q\in\{a,b,c\}\), \(p\ne q\), and an index \(s\) such that \(z_kw_t-p_sq_s\in \Ia\). Therefore, we have \(x_iy_jz_kw_t= (x_iy_j)(z_kw_t-p_sq_s) + x_iy_jp_sq_s\equiv x_iy_jp_sq_s\). Hence, by symmetry, it suffices to examine monomials of the form \(a_ib_ix_jy_k\) with \(x_jy_k\in \Ia\).

If any of \(j,k\) are in \(\{i-1,i+1\}\), then Lemma \ref{lem_i_plus_one} applies. If both \(j\) and \(k\) are not in \(\{i-1,i,i+1\}\), then \(a_ib_ix_jy_k=(a_ix_j)(b_iy_k)\equiv 0\). Thus, we can assume one of \(j,k\) equal to \(i\), say \(k=i\).

Now, if \(y\ne c\), since \(x_jy_i\in \Ia\), we get \(a_ib_ix_jy_i=(a_ib_i-c_{i-1}c_i)(x_jy_i) + c_{i-1}y_ic_ix_j\), hence Lemma \ref{lem_i_plus_one} applies. We are therefore reduced to monomials of the form \(a_ib_ic_ix_j\). By symmetry we can assume \(x=a\), and use Lemma \ref{lem_special_cases} to finish the proof.
\end{proof}

\begin{proposition} \label{prop_0}
Every degree \(4\) homogenous polynomial of \(\Sa\) is contained in \(\Ja\).
\end{proposition}

\begin{proof}
In view of Lemma \ref{lem_\Ia} it suffices to verify monomials where no two indices differ by more than 1. Moreover, if two indices differ exactly by 1, and not all letters are equal, then Lemma \ref{lem_i_plus_one} applies. Hence, by symmetry, it suffices to examine \(a_i^2b_ic_i, a_i^2b_i^2, a_i^3b_i, a_i^4,  a_i^3a_{i+1}, a_ia_{i+1}^3\), and \(a_i^2a_{i+1}^2\). Clearly, \(a_i^2b_ic_i\) is in \(\Ja\). Monomial \(a_i^4\) was shown to be in \(\Ia^2\) in the proof of Lemma \ref{lem_special_cases}, hence, by symmetry, \(c_i^4\) is also in \(\Ia^2\). We can rewrite the remaining monomials as follows.
\begin{alignat*}{2}
a_i^2b_i^2 & = (a_{i}b_{i}-c_{i+1}^2)^2 +2(a_ic_{i+1})(b_ic_{i+1}) - c_{i+1}^4 \equiv 0 \\
a_i^3b_i & = (a_ib_i-c_{i+1}^2)(a_i^2-b_{i-1}c_{i-1}) +(a_ib_{i-1})(b_{i}c_{i-1}) +(a_ic_{i+1})^2 + \\
 & -(c_{i-1}c_{i+1})(b_{i-1}c_{i+1}) \equiv 0 \\
a_i^3a_{i+1} & = (a_i^2-b_{i-1}c_{i-1})(a_ia_{i+1}-b_{i+1}c_{i+1}) + a_i^2b_{i+1}c_{i+1} + \\
 & + a_ia_{i+1}b_{i-1}c_{i-1} - (b_{i-1}b_{i+1})(c_{i-1}c_{i+1}) \\
a_ia_{i+1}^3 & = (a_ia_{i+1}-b_{i+1}c_{i+1})(a_{i+1}^2-b_ic_i) + a_ia_{i+1}b_ic_i + \\ 
 & + a_{i+1}^2b_{i+1}c_{i+1} - b_ib_{i+1}c_ic_{i+1} \\
a_i^2a_{i+1}^2 & = (a_i^2-a_{i-2}a_{i-1})(a_{i+1}^2-b_ic_i) + a_i^2b_ic_i + (a_{i-2}a_{i+1})(a_{i-1}a_{i+1}) + \\
 & - (a_{i-2}b_i)(a_{i-1}c_i-a_{i+1}c_{i+2}) - (a_{i-2}a_{i+1})(b_ic_{i+2})
\end{alignat*}
Then, Lemmas \ref{lem_special_cases} and \ref{lem_i_plus_one} finish the proof.
\end{proof}

\begin{lemma} \label{lem_five_0}
For all \(x,y\in\{a,b,c\}\), and for all indices \(i,j\), the monomial \(x_iy_ja_jb_jc_j\) is in \(\Ia^2\). 
\end{lemma}

\begin{proof}
By symmetry we can assume \(y=a\). Note that \(a_j^2b_j\) annihilates \(\Fa\), so is in \(\Ia\). If \(i\notin\{j-1,j,j+1\}\), then \(x_ic_j\in \Ia\), so \(x_ia_j^2b_jc_j=(x_ic_j)(a_j^2b_j)\in \Ia^2\). Now suppose \(i\in\{j-1,j,j+1\}\). Either \(x\ne b\) or \(x\ne c\), by symmetry we can assume \(x\ne c\). Then, in the case \(i\in\{j-1,j+1\}\), we obtain \(x_ia_j^2b_jc_j=(x_ic_j-x_{i+3}c_{j+3})(a_j^2b_j)+(x_{i+3}a_j)(a_jc_{j+3})b_j \in \Ia^2\). In the case \(i=j\) we need to consider monomials \(a_i^3b_ic_i\) and \(a_i^2b_i^2c_i\). We rewrite them as follows.
\begin{alignat*}{2}
a_i^3b_ic_i & = (a_i^2b_i)(a_ic_i-b_{i-1}b_i) + (a_ib_i^2)(a_ib_{i-1}-a_{i-2}b_{i-3}) + (a_{i-2}a_i)(b_{i-3}b_i)b_i \equiv 0 \\
a_i^2b_i^2c_i & = (a_i^2c_i)(b_i^2-b_{i-2}b_{i-1}) + (a_i^2b_{i-1})(b_{i-2}c_i)\equiv 0
\end{alignat*}
This finishes the proof.
\end{proof}

\begin{proposition} \label{prop_done_0}
The apolar algebra \(\Apolar(\Fa)\) satisfies the small tangent space condition.
\end{proposition}

\begin{proof}
It is easy to check that no linear form annihilates \(\Fa\), hence \(\Apolar(\Fa)\) has Hilbert function \((1,n,n,1)\). Proposition \ref{prop_0} implies that \(H(\Sa/\Ia^2)_4\leq n\), so \(H(\Sa/\Ann(\Fa)^2)\leq n\). Thus, by Proposition \ref{prop_low_bound}, \(H(\Sa/\Ann(\Fa)^2)_4=n\). Finally, since monomials of the form \(x_ia_ib_ic_i\) generate \((\Sa/\Ia^2)_4\) Lemma \ref{lem_five_0} implies that \(H(\Sa/\Ia^2)_5=0\), so also \(H(\Sa/\Ann(\Fa)^2)=0\).
\end{proof}

\subsection{Proof of Theorem \ref*{thm_main}; case \(n=3m+1\)}
\hfill

Let \(\Sb=\bk[a_i,b_i,c_i,d]_{i=1}^m\) be a polynomial ring of \(n=3m+1\) variables. Recall that we assume \(m\geq 6\). When writing indices we treat them modulo \(m\).

Consider the following polynomial
\[\Fb=\sum_{i=1}^m a_ib_ic_i + a_ia_{i+1}^2 + b_ib_{i+1}^2 + c_ic_{i+1}^2 + a_ib_{i+1}d.\]
Also by \(\Fb\) we denote its dual element in the divided power ring associated to \(\Sb\).

Let \(\Ib\) be the smallest ideal such that the following remark holds.

\begin{remark}
For all \(x,y\in\{a,b,c\}\), and for all indices \(i,j\), if \(j\notin\{i-1,i,i+1\}\), then \(x_iy_j\in \Ib\).

For all \(x,y\in\{a,b,c\}\), and for all indices \(i,j\), if \(x\ne y\), then \(x_iy_{i+1}-x_jy_{j+1}\in \Ib\).

For all \(x,y\in\{a,b,c\}\), and for all indices \(i,j\), one of the following holds.
\begin{enumerate}
\item \(x_iy_j\) is contained in \(\Ib\).
\item \(x\ne y\) and \(j\in\{i-1,i+1\}\).
\item There exists an index \(k\) and \(p,q\in\{a,b,c\}\), \(p\ne q\), such that \(x_iy_j-p_kq_k\in \Ib\).
\end{enumerate}

For any \(x\in\{a,b,c\}\), and any index \(i\), one of the following holds.
\begin{enumerate}
\item \(x_id\) is contained in \(\Ib\).
\item There exists an index \(j\) and \(p,q\in\{a,b,c\}\), \(p\ne q\), such that \(x_id-p_jq_j\in \Ib\).
\end{enumerate}
\end{remark}

The polynomial \(\Fb\) is chosen such that \(\Ib\subset\Ann\,\Fb\).

Let \(\Jb\) denote the ideal \((\Ib)^2 + \langle a_ia_{i+1}a_{i+2}^2,b_ib_{i+1}b_{i+2}^2,c_ic_{i+1}c_{i+2}^2 \,|\, i=1,...,m\rangle+\langle a_1b_1c_1d\rangle\). Note that \(\Fb,\Ib\), and \(\Jb\) are invariant under index translation.

We want to show that \(\Apolar(\Fb)\) satisfies the small tangent space condition. The main part of the proof is checking that all polynomials of degree \(4\) are contained in \(\Jb\), hence that \(H(\Sb/(\Ib)^2)_4=n\).

In this section, we use the following notation. For polynomials \(Q,R\in \Sb\) we write \(Q\equiv R\) if \(Q\) is equal to \(R\) in \(\Sb/(\Ib)^2\).

\begin{lemma} \label{lem_no_d}
All monomials of degree \(4\), not divisible by \(d\) are contained in \(\Jb\).
\end{lemma}

\begin{proof}
We have a natural inclusion of rings \(\Sa\subset \Sb\). Note that \(\Ia\subset \Ib\cap \Sa\), so also \(\Ja\subset \Jb\cap \Sa\). Since every monomial not divisible by \(d\) is contained in \(\Sa\), the claim follows from Proposition \ref{prop_0}.
\end{proof}

\begin{lemma} \label{lem_1d}
For all \(x,y,z\in\{a,b,c\}\), and for all indices \(i,j,k\), the monomial \(x_iy_jz_kd\) is contained in \(\Jb\).
\end{lemma}

\begin{proof}
If any of \(x_iy_j\), \(x_iz_k\), \(y_jz_k\) is in \(\Ib\), say \(x_iy_j\in \Ib\), then either \(z_kd\in \Ib\), so that \(x_iy_jz_kd\in (\Ib)^2\), or there exist \(p,q\in\{a,b,c\}\) and an index \(t\) such that \(z_kd-p_tq_t\in \Ib\), so \(x_iy_jz_kd= (x_iy_j)(z_kd-p_tq_t)+x_iy_jp_tq_t\). Monomial \(x_iy_jp_tq_t\) is contained in \(\Jb\) by Lemma \ref{lem_no_d}.

If any of \(x_iy_j\), \(x_iz_k\), \(y_jz_k\) is of the form \(w_tv_{t+1}\), \(w\ne v\), say \(j=i+1\), \(x\ne y\), then either \(z_kd\in \Ib\) and we can rewrite \(x_iy_{i+1}z_kd=(x_iy_{i+1}-x_{k+1}y_{k+2})(z_kd) + x_{k+1}y_{k+2}z_{k}d\), or \(z_kd\notin \Ib\) and there exist \(p,q\in\{a,b,c\}\) and an index \(t\) such that \(z_kd-p_tq_t\in \Ib\), hence \(x_iy_{i+1}z_kd=(x_iy_{i+1}-x_{k+1}y_{k+2})(z_kd-p_tq_{t}) + x_{k+1}y_{k+2}z_{k}d + x_iy_{i+1}p_tq_{t} - x_{k+1}y_{k+2}p_tq_t\). In any case the claim follows by the previous paragraph and Lemma \ref{lem_no_d}.

It remains to consider the case where either \(x=y=z\) and \(j,k\in\{i,i+1\}\), or \(i=j=k\). We first consider the case where \(i=j=k\) and not all \(x,y,z\) are the same. If \(x,y,z\) are not mutually different, say \(x=y\), then since \(y\ne z\) there exists \(w\in\{a,b,c\}\), \(w\ne x\), such that \(y_iz_i-w_{i-1}w_i\in \Ib\). Hence, if \(x_id\in \Ib\) we get \(x_iy_iz_id = (x_id)(y_iz_i-w_{i-1}w_i) + x_iw_{i-1}w_id\), and if \(x_id\notin \Ib\), then there are \(p,q\in\{a,b,c\}\) and an index \(s\) such that \(x_id-p_sq_s\in \Ib\), so \(x_iy_iz_id = (x_id-p_sq_s)(y_iz_i-w_{i-1}w_i) + x_iw_{i-1}w_id + y_iz_ip_sq_s - w_{i-1}w_ip_sq_s\). Thus, the claim follows by the previous parts of the proof and Lemma \ref{lem_no_d}. Now we consider the case where \(x,y,z\) are mutually different, hence we need to examine the monomial \(a_ib_ic_id\). We rewrite it as follows.
\begin{alignat*}{2}
a_ib_ic_id & = (a_ib_i-c_{i+1}^2)(c_id) + (c_ic_{i+1}-a_{i+1}b_{i+1})(c_{i+1}d) + a_{i+1}b_{i+1}c_{i+1}d
\end{alignat*}
Therefore, \(a_ib_ic_id\equiv a_{i+1}b_{i+1}c_{i+1}d\), and so \(a_ib_ic_id \equiv a_1b_1c_1d\in \Jb\).

We now consider the case where \(x=y=z\). If \(i=j=k\), then there are three monomials to consider, \(a_i^3d\), \(b_i^3d\), and \(c_i^3d\). We rewrite them in the following way.
\begin{alignat*}{2}
a_i^3d & = (a_i^2-b_{i-1}c_{i-1})(a_id-a_{i+1}c_{i+1}) + a_i^2a_{i+1}c_{i+1} + \\
 & + (a_ib_{i-1})(c_{i-1}d) - (a_{i+1}b_{i-1})(c_{i-1}c_{i+1}) \equiv 0 \\
b_i^3d & = (b_i^2-a_{i-1}c_{i-1})(b_id-a_i^2) + a_i^2b_i^2 + a_{i-1}b_ic_{i-1}d + \\
& - a_{i-1}a_i^2c_{i-1} \equiv 0 \\
c_i^3d & = (c_i^2-c_{i-2}c_{i-1})(c_id) + (c_{i-2}c_i)(c_{i-1}d) \equiv 0
\end{alignat*}
Hence, Lemma \ref{lem_no_d} and previous parts of the proof apply. Now suppose \(x=y=z\) and at least one of \(j,k\) is \(i+1\), say \(j= i+1\), hence there exist \(p,q\in\{a,b,c\}\), \(p\ne q\), and an index \(t\) such that \(x_ix_{i+1}-p_tq_t\in \Ib\). Then, either \(x_kd\in \Ib\) and we get \(x_ix_{i+1}x_kd=(x_ix_{i+1}-p_tq_t)(x_kd) + x_kp_tq_td\), or there exist \(w,v\in\{a,b,c\}\) and an index \(s\) such that \(x_kd-w_sv_s\in \Ib\), so we get \(x_ix_{i+1}x_kd=(x_ix_{i+1}-p_tq_t)(x_kd-w_sv_s) + x_ix_{i+1}w_sv_s+x_kp_tq_td-p_tq_tw_sv_s\). Either case follows from the previous parts of the proof and Lemma \ref{lem_no_d}.
\end{proof}

\begin{lemma} \label{lem_2d}
For all \(x,y\in\{a,b,c\}\), and all indices \(i,j\), the monomial \(x_iy_jd^2\) is contained in \(\Jb\).
\end{lemma}

\begin{proof}
If both \(x_id\) and \(y_jd\) are in \(\Ib\), then \(x_iy_jd^2\in (\Ib)^2\). If only one of them is in \(\Ib\), say \(x_id\in \Ib\), then there exist \(z,w\in\{a,b,c\}\) and an index \(k\) such that \(y_jd-z_kw_k\in \Ib\), so \(x_iy_jd^2 = (x_id)(y_jd-z_kw_k) + x_iz_kw_kd\). Finally, if \(x_id\notin \Ib\), \(y_jd\notin \Ib\), then there exist \(w,z,p,q\in\{a,b,c\}\) and indices \(k,t\) such that \(x_id-w_kz_k\in \Ib\) and \(y_jd-p_tq_t\in \Ib\), so \(x_iy_jd^2 = (x_id-p_tq_t)(y_jd-z_kw_k) + x_iz_kw_kd + y_jp_tq_td - z_kw_kp_tq_t\). Hence, the claim follows from Lemmas \ref{lem_no_d} and \ref{lem_1d}.
\end{proof}

\begin{proposition} \label{prop_1}
Every degree \(4\) homogenous polynomial of \(\Sb\) is contained in \(\Jb\).
\end{proposition}

\begin{proof}
Lemmas \ref{lem_no_d}, \ref{lem_1d}, and \ref{lem_2d} cover most of the cases. It remains to check that for any \(x\in\{a,b,c\}\) and any index \(i\) both \(x_id^3\) and \(d^4\) are in \(\Jb\). Thus, we need to consider four monomials, which we rewrite as follows.
\begin{alignat*}{2}
a_id^3 & = (a_id-a_{i+1}c_{i+1})(d^2) + a_{i+1}c_{i+1}d^2 \\
b_id^3 & = (b_id-b_{i-1}c_{i-1})(d^2) + b_{i-1}c_{i-1}d^2 \\
c_id^3 & = (c_id)(d^2) \equiv 0 \\
d^4 & = (d^2)^2 \equiv 0
\end{alignat*}
Thus, Lemma \ref{lem_2d} finishes the proof.
\end{proof}

\begin{lemma} \label{lem_five_1}
For all \(x,y\in\{a,b,c\}\), and for all indices \(i,j\), the monomials \(x_iy_ja_jb_jc_j\), \(x_ia_jb_jc_jd\) and \(a_jb_jc_jd^2\) are in \((\Ib)^2\).
\end{lemma}

\begin{proof}
That \(x_iy_ja_jb_jc_j\) is in \((\Ib)^2\) follows from the inclusion of rings \(\Sa\subset \Sb\) and Lemma \ref{lem_five_0}. Note that \(a_jb_jd\) annihilates \(\Fb\), so is in \(\Ib\). If \(x_i\ne c_j\), then \(x_ia_jb_j\in \Ib\), so \(x_ia_jb_jc_jd=(x_ia_jb_j)(c_jd)\in (\Ib)^2\). Thus, it remains to consider \(a_jb_jc_j^2d\) and \(a_jb_jc_jd^2\). We rewrite them as follows.
\begin{alignat*}{2}
a_jb_jc_j^2d & = (a_jb_jd)(c_j^2-c_{j-2}c_{j-1}) + (a_jc_{j-2})(c_{j-1}d)b_j \equiv 0 \\
a_jb_jc_jd^2 & = (a_jb_jd)(c_jd) \equiv 0
\end{alignat*}
Hence, the proof is complete.
\end{proof}

\begin{proposition} \label{prop_done_1}
The apolar algebra \(\Apolar(\Fb)\) satisfies the small tangent space condition.
\end{proposition}

\begin{proof}
It is easy to check that no linear form annihilates \(\Fb\), hence \(\Apolar(\Fb)\) has Hilbert function \((1,n,n,1)\). Proposition \ref{prop_1} implies that \(H(\Sb/(\Ib)^2)_4\leq n\), so \(H(\Sb/\Ann(\Fb)^2)\leq n\). Thus, by Proposition \ref{prop_low_bound}, \(H(\Sb/\Ann(\Fb)^2)_4=n\). Finally, since monomials of the form \(x_ia_ib_ic_i\) and \(a_1b_1c_1d\) generate \((\Sb/(\Ib)^2)_4\) Lemma \ref{lem_five_1} implies that \(H(\Sb/(\Ib)^2)_5=0\), so also \(H(\Sb/\Ann(\Fb)^2)=0\).
\end{proof}

\subsection{Proof of Theorem \ref*{thm_main}; case \(n=3m+2\)}
\hfill

Let \(\Sc=\bk[a_i,b_i,c_i,d,e]_{i=1}^m\) be a polynomial ring of \(n=3m+2\) variables. Recall that we assume \(m\geq 6\). When writing indices we treat them modulo \(m\).

Consider the following polynomial
\[\Fc=\sum_{i=1}^m a_ib_ic_i + a_ia_{i+1}^2 + b_ib_{i+1}^2 + c_ic_{i+1}^2 + a_ib_{i+1}d + b_ic_{i+1}e.\]
Also by \(\Fc\) we denote its dual element in the divided power ring associated to \(\Sc\).

Let \(\Ic\) be the smallest ideal such that the following remark holds.

\begin{remark}
For all \(x,y\in\{a,b,c\}\), and for all indices \(i,j\), if \(j\notin\{i-1,i,i+1\}\), then \(x_iy_j\in \Ic\).

For all \(x,y\in\{a,b,c\}\), and for all indices \(i,j\), if \(x\ne y\), then \(x_iy_{i+1}-x_jy_{j+1}\in \Ic\).

For all \(x,y\in\{a,b,c\}\), and for all indices \(i,j\), one of the following holds.
\begin{enumerate}
\item \(x_iy_j\) is contained in \(\Ic\).
\item \(x\ne y\) and \(j\in\{i-1,i+1\}\).
\item There exists an index \(k\) and \(p,q\in\{a,b,c\}\), \(p\ne q\), such that \(x_iy_j-p_kq_k\in \Ic\).
\end{enumerate}

For any \(x\in\{a,b,c\}\), any \(y\in\{d,e\}\), and any index \(i\), one of the following holds.
\begin{enumerate}
\item \(x_iy\) is contained in \(\Ic\).
\item There exists an index \(j\) and \(p,q\in\{a,b,c\}\), \(p\ne q\), such that \(x_iy-p_jq_j\in \Ic\).
\end{enumerate}
\end{remark}

The polynomial \(\Fc\) is chosen such that \(\Ic\subset\Ann\,\Fc\).

Let \(\Jc\) denote the ideal \((\Ic)^2 + \langle a_ia_{i+1}a_{i+2}^2,b_ib_{i+1}b_{i+2}^2,c_ic_{i+1}c_{i+2}^2 \,|\, i=1,...,m\rangle+\langle a_1b_1c_1d,a_1b_1c_1e\rangle\). Note that \(\Fc,\Ic\), and \(\Jc\) are invariant under index translation.

We want to show that \(\Apolar(\Fc)\) satisfies the small tangent space condition. The main part of the proof is checking that all polynomials of degree \(4\) are contained in \(\Jc\), hence that \(H(\Sc/(\Ic)^2)_4=n\).

In this section, we use the following notation. For polynomials \(Q,R\in \Sc\) we write \(Q\equiv R\) if \(Q\) is equal to \(R\) in \(\Sc/(\Ic)^2\).

\begin{lemma} \label{lem_no_de}
All monomials of degree \(4\), not divisible by \(de\) are contained in \(\Jc\).
\end{lemma}

\begin{proof}
We have two inclusions of rings \(\Sb\subset \Sc\), one takes \(a,b,c,d\) to \(a,b,c,d\) respectively, the other takes \(a,b,c,d\) to \(b,c,a,e\) respectively. Note that, in both cases, \(\Ib\subset \Ic\cap \Sb\), so \(\Jb\subset \Jc\cap \Sb\). Since every monomial not divisible by \(de\) is contained in at least one of those subrings, the claim follows from Proposition~\ref{prop_1}.
\end{proof}

\begin{lemma} \label{lem_de}
For all \(x,y\in\{a,b,c\}\), and for all indices \(i,j\), the monomial \(x_iy_jde\) is contained in \(\Jc\).
\end{lemma}

\begin{proof}
If both \(x_id\) and \(y_je\) are in \(\Ic\), then \(x_iy_jde\in (\Ic)^2\). If only \(x_id\) is in \(\Ic\), then there exist \(z,w\in\{a,b,c\}\) and an index \(k\) such that \(y_je-z_kw_k\in \Ic\), so \(x_iy_jde = (x_id)(y_je-z_kw_k) + x_iz_kw_kd\). Similarly, if only \(y_je\) is in \(\Ic\), then there exist \(z,w\in\{a,b,c\}\) and an index \(k\) such that \(x_id-z_kw_k\in \Ic\), so \(x_iy_jde=(x_id-z_kw_k)(y_je)+ y_jz_kw_ke\). If both \(x_id\) and \(y_je\) are not in \(\Ic\), then there exist \(w,z,p,q\in\{a,b,c\}\) and indices \(k,t\) such that \(x_id-w_kz_k\in \Ic\) and \(y_jd-p_tq_t\in \Ic\), so \(x_iy_jde = (x_id-p_tq_t)(y_je-z_kw_k) + x_iz_kw_kd + y_jp_tq_te - z_kw_kp_tq_t\). Hence, the claim follows from Lemma \ref{lem_no_de}.
\end{proof}

\begin{proposition} \label{prop_2}
Every degree \(4\) homogenous polynomial of \(\Sc\) is contained in \(\Jc\).
\end{proposition}

\begin{proof}
Lemmas \ref{lem_no_de} and \ref{lem_de} cover most of the cases. The rest we rewrite as follows.
\begin{alignat*}{2}
a_id^2e & = (a_ie)(d^2) \equiv 0 \\
b_id^2e & = (b_ie-c_{i+2}^2)(d^2) + (c_{i+2}d)^2 \equiv 0 \\
c_id^2e & = (c_id)(de) \equiv 0 \\
a_ide^2 & = (a_ie)(de) \equiv 0 \\
b_ide^2 & = (b_id-a_i^2)(e^2) + (a_ie)^2 \equiv 0 \\
c_ide^2 & = (c_id)(e^2)\equiv 0 \\
d^3e & = (d^2)(de) \equiv 0 \\
d^2e^2 & = (de)^2 \equiv 0 \\
de^3 & = (de)(e^2) \equiv 0
\end{alignat*}
This finishes the proof.
\end{proof}

\begin{lemma} \label{lem_five_2}
For all \(x,y\in\{a,b,c\}\), \(z,w\in \{d,e\}\), and for all indices \(i,j\), the monomials \(x_iy_ja_jb_jc_j\), \(x_ia_jb_jc_jz\) and \(a_jb_jc_jzw\) are in \((\Ic)^2\).
\end{lemma}

\begin{proof}
We have two inclusions of rings \(\Sb\subset \Sc\), one takes \(a,b,c,d\) to \(a,b,c,d\) respectively, the other takes \(a,b,c,d\) to \(b,c,a,e\) respectively. Hence, in view of Lemma \ref{lem_five_1} it suffices to consider \(a_jb_jc_jde\). We have \(a_jb_jc_jde=(a_je)(c_jd)b_j\in (\Ic)^2\).
\end{proof}

\begin{proposition} \label{prop_done_2}
The apolar algebra \(\Apolar(\Fc)\) satisfies the small tangent space condition.
\end{proposition}

\begin{proof}
It is easy to check that no linear form annihilates \(\Fc\), hence \(\Apolar(\Fc)\) has Hilbert function \((1,n,n,1)\). Proposition \ref{prop_2} implies that \(H(\Sc/(\Ic)^2)_4\leq n\), so \(H(\Sc/\Ann(\Fc)^2)\leq n\). Thus, by Proposition \ref{prop_low_bound}, \(H(\Sc/\Ann(\Fc)^2)_4 =n\). Finally, since monomials of the form \(x_ia_ib_ic_i\), \(a_1b_1c_1d\), and \(a_1b_1c_1e\) generate \((\Sc/(\Ic)^2)_4\) Lemma \ref{lem_five_2} implies that \(H(\Sc/(\Ic)^2)=0\), so also \(H(\Sc/\Ann(\Fc)^2)=0\).
\end{proof}

\section{Computer computations for \(n<18\)} \label{chap_computer}

Let \(S\) be a polynomial ring of \(n\) variables. In this chapter, we give examples of degree \(3\) polynomials \(F\) such that \(\Apolar(F)\) satisfies the small tangent space condition for \(n=6\) and \(7<n<18\) (the case \(n\geq18\) is covered by Theorem \ref{thm_main}).

We have checked on computer, using Macaulay2 \cite{M2}, that they are indeed correct for fields of characteristic 0, 2, and 3. We believe that they work in any characteristic, though a proof would probably require a direct verification, so we restrict ourselves to supplying a computer code which one can use to verify these examples in any given characteristic.

Note that in order to verify Conjecture \ref{conj} for a field \(\bk\) of characteristic \(0\) it suffices to check \(\bk=\mathbb Q\). Similarly, for a field \(\bk\) of characteristic \(p\) is suffices to check \(\bk=\mathbb F_p\).

Our examples from chapter \ref{chp_theorem} work also for \(n\geq 9\). For \(n=6\) and \(n=8\) we construct different polynomials. For \(n=6\) we have chosen the polynomial
\[F=a_1b_1c_1 + a_2b_2c_2 + a_1a_2^2 + b_1b_2^2 + c_1c_2^2 + a_1^3 + b_1^3 + c_1^3.\]
For \(n=8\) we have chosen
\[F=a_1b_1c_1 + a_2b_2c_2 + a_1a_2^2 + b_1b_2^2 + c_1c_2^2 + a_1de + b_1^2d + c_1^2e.\]

\subsection{Macaulay2 code}
\hfill

In this section, we describe the computer code we have used to verify our examples. First one needs to chose a field, hence to type
\begin{verbatim}
  kk = QQ;
\end{verbatim}
or (replacing \(p\) by a prime number of choice)
\begin{verbatim}
  kk = ZZ/p;
\end{verbatim}
into the Macaulay2 console. Then, one needs to specify the number of variables of the polynomial ring typing
\begin{verbatim}
  n = ?
\end{verbatim}
with \(?\) replaced by the chosen number. If \(n\) was chosen to be \(6\), then the following code generates the appropriate polynomial.
\begin{verbatim}
  S = kk[a_1,a_2,b_1,b_2,c_1,c_2];
  F = a_1*b_1*c_1 + a_2*b_2*c_2 + a_1*a_2^2 + b_1*b_2^2 +
      c_1*c_2^2 + a_1^3 + b_1^3 + c_1^3;
\end{verbatim}
If \(n=8\) one needs to enter the following lines.
\begin{verbatim}
  S = kk[a_1,a_2,b_1,b_2,c_1,c_2,d,e];
  F = a_1*b_1*c_1 + a_2*b_2*c_2 + a_1*a_2^2 + b_1*b_2^2 +
      c_1*c_2^2 + a_1*d*e + b_1^2*d + c_1^2*e;
\end{verbatim}
If \(n\) is divisible by \(3\) and greater than \(8\), then the following code needs to be entered.
\begin{verbatim}
  m = n//3;
  S = kk[a_1..a_m,b_1..b_m,c_1..c_m];
  F = 0;
  for i in 1..m-1 do F = F + a_i*b_i*c_i + a_i*a_(i+1)^2 +
      b_i*b_(i+1)^2 + c_i*c_(i+1)^2;
  F = F + a_m*b_m*c_m + a_m*a_1^2 + b_m*b_1^2 + c_m*c_1^2;
\end{verbatim}
If \(n\) gives remainder \(1\) upon division by \(3\) and is greater than \(8\), then one uses the following code.
\begin{verbatim}
  m = (n-1)//3;
  S = kk[a_1..a_m,b_1..b_m,c_1..c_m,d];
  F = 0;
  for i in 1..m-1 do F = F + a_i*b_i*c_i + a_i*a_(i+1)^2 +
      b_i*b_(i+1)^2 + c_i*c_(i+1)^2 + a_i*b_(i+1)*d;
  F = F + a_m*b_m*c_m + a_m*a_1^2 + b_m*b_1^2 + c_m*c_1^2 +
      a_m*b_1*d;
\end{verbatim}
Lastly, if \(n\) gives remainder \(2\) upon division by \(3\) and is greater than \(8\), then the following code needs to be used.
\begin{verbatim}
  m = (n-2)//3;
  S = kk[a_1..a_m,b_1..b_m,c_1..c_m,d,e];
  F = 0;
  for i in 1..m-1 do F = F + a_i*b_i*c_i + a_i*a_(i+1)^2 +
      b_i*b_(i+1)^2 + c_i*c_(i+1)^2 + a_i*b_(i+1)*d +
      b_i*c_(i+1)*e;
  F = F + a_m*b_m*c_m + a_m*a_1^2 + b_m*b_1^2 + c_m*c_1^2 +
      a_m*b_1*d + b_m*c_1*e;
\end{verbatim}
To verify whether the apolar algebra induced by the generated polynomial satisfies the small tangent space condition one can run the following lines.
\begin{verbatim}
  I = ideal fromDual(matrix{{F}}, DividedPowers => true);
  if (hilbertFunction(0,S/I) == 1 and
      hilbertFunction(1,S/I) == n and
      hilbertFunction(4,S/I^2) == n and
      hilbertFunction(5,S/I^2) == 0)
      then print True else print False;
\end{verbatim}
If the answer given by Macaulay2 reads "True", then \(\Apolar(F)\) satisfies the small tangent space condition. If on the other hand the answer reads "False", then \(\Apolar(F)\) does not satisfy the small tangent space condition.

\bibliographystyle{alpha}
\bibliography{references}

\end{document}